\documentclass[12pt,electronic]{amsart}                                                 
\usepackage[T1]{fontenc}
\usepackage[applemac]{inputenc}
\usepackage[english]{babel}
\usepackage[small,nohug]{diagrams}
\newarrow {Equals} =====
\usepackage{url}
\usepackage[hmargin=3.9cm,vmargin=3.5cm]{geometry} 
\usepackage{amsmath}
\usepackage{amsthm}                      
\usepackage{amssymb}
\usepackage{mathrsfs}
\usepackage{enumerate}
\theoremstyle{plain}                                                           
\newtheorem{thm}{Theorem}[section]
\newtheorem*{thma}{Theorem \ref{aa}}
\newtheorem*{thmb}{Theorem \ref{bb}}
\newtheorem{speculation}[thm]{Speculation}
\newtheorem{lem}[thm]{Lemma}
\newtheorem{prop}[thm]{Proposition}
\newtheorem{cor}[thm]{Corollary}
\newtheorem{conjecture}[thm]{Conjecture}
\theoremstyle{definition}

\newtheorem{rem}[thm]{Remark}

\newcommand{\Eis}{\mathrm{Eis}}
\DeclareMathOperator{\Sym}{Sym}

\DeclareMathOperator{\Sp}{Sp}

\DeclareMathOperator{\SL}{SL}
\DeclareMathOperator{\Frob}{Frob}
\newcommand{\Gal}{\mathsf{Gal}}

\newcommand{\field}[1]{\ensuremath{\mathbf{#1}}}
\newcommand{\Q}{\ensuremath{\field{Q}}}        
\newcommand{\sym}{\ensuremath{\mathbb{S}}}
\newcommand{\Z}{\ensuremath{\field{Z}}}

\newcommand{\U}{\mathbb{U}_0^-}
\newcommand{\FM}{\mathsf{FM}}

\newcommand{\M}{\mathcal{M}}
\newcommand{\C}{\mathcal{C}}
\newcommand{\D}{\mathcal{D}}
\newcommand{\F}{\mathbb{W}}
\newcommand{\A}{\mathcal{A}}
\newcommand{\MM}{\overline{\mathcal{M}}}

\newcommand{\gr}{\mathfrak{gr}}
\newcommand{\V}{\mathbb{V}}
\newcommand{\X}{\mathcal{X}}
\newcommand{\R}{\mathrm{R}}       
\newcommand{\ct}{\mathit{ct}}
\newcommand{\rt}{\mathit{rt}}

\title{The Gorenstein conjecture fails for the tautological ring of $\MM_{2,n}$}
\author{Dan Petersen}
\email{danpete@math.kth.se}
\address{Department of Mathematics \\ KTH Royal Institute of Technology \\ 100 44 Stockholm \\ Sweden}
\thanks{The first named author is supported by the G\"oran Gustafsson foundation for scientific and medical research. The second named author is supported by DFG under grant Hu337/6-2.}
\author{Orsola Tommasi}
\email{tommasi@math.uni-hannover.de}
\address{Institut für Algebraische Geometrie \\
Leibniz Universität Hannover \\
Welfengarten 1 \\
D-30167 Hannover \\
Germany}

\begin{document} 
  
 \maketitle

\begin{abstract} Let $N$ be the smallest integer such that there is a non-tautological cohomology class of even degree on $\MM_{2,N}$. We remark that there is such a non-tautological class on $\MM_{2,20}$, by work of Graber and Pandharipande. We show that $\MM_{2,N}$ has non-tautological cohomology only in one degree, which is \emph{not} the middle degree. In particular, it follows that the tautological ring $R^\bullet(\MM_{2,N})$ is not Gorenstein. We present some evidence suggesting that $N=20$ holds. \end{abstract}

\section{Introduction} 

The tautological ring of $\M_g$ was introduced in \cite{mumfordtowards} as a natural subring $R^\bullet(\M_g)$ of the rational Chow ring containing most ``geometrically meaningful'' classes. Mumford originally defined it as the subalgebra generated by the $\kappa$-classes. By working with moduli of stable curves with marked points, one can give a more general and economical definition, as proposed in \cite{relativemaps}: the system of all tautological rings $\{R^\bullet(\MM_{g,n})\}$ is the smallest collection of $\Q$-subalgebras of  $\{A^\bullet(\MM_{g,n})\}$ closed under pushforward along gluing and forgetful morphisms. If $U \subset \MM_{g,n}$ is Zariski open, one defines the tautological ring of $U$ to be the image of the natural restriction map. Finally, one also defines the tautological cohomology ring $RH^\bullet(\MM_{g,n})$ to be the image of the ordinary tautological ring in cohomology. 

Based on empirical observations, Faber \cite{faberconjectures} formulated an ambitious conjecture giving a precise description of the structure of the tautological ring of $\M_g$. Most of these conjectures have by now been proven by work of a large amount of people, see \cite{pcmifaber} for a survey. But the most elusive part of the conjecture has turned out to be the ``Gorenstein'' part of the conjecture, which asserts that $R^\bullet(\M_g)$ is a Gorenstein ring with socle in dimension $g-2$ (in other words, that the ring satisfies Poincar\'e duality).  Analogous Gorenstein conjectures have been formulated also for the spaces $\MM_{g,n}$ (stable $n$-pointed curves), $\M_{g,n}^\ct$ (curves of compact type) and $\M_{g,n}^\rt$ (curves with rational tails), see \cite{faberjapan,pandharipandequestions}.  It is known for all these spaces that the tautological ring vanishes above the expected degree and that its top degree is one-dimensional, by the work of \cite{looijengatautological, logarithmicseries,gv1,gv2}.
  
 However, there is by now some evidence suggesting that the Gorenstein property may be false in general. No known method of constructing relations between tautological classes will at this point produce enough relations between the generators of $R^\bullet(\M_{24})$ to yield a Gorenstein ring, as discussed in \cite[Lecture 1]{pcmifaber}. In particular, the conjecture that the Faber--Zagier relations give rise to all relations in the tautological ring \cite{pandapixton} implies that $R^\bullet(\M_{24})$ is not Gorenstein. Yin \cite{yin} finds evidence that $R^\bullet(\M_{20,1})$ is not Gorenstein. Some weaker circumstantial evidence is provided in \cite{cavalieriyang}. 
 
 In this paper we find the first ``smoking gun'', by proving that the tautological ring of $\MM_{2,n}$ is not always Gorenstein. Our first observation is that if the tautological ring is Gorenstein, then the cycle class map to the tautological cohomology ring is an isomorphism onto its image, and so the tautological cohomology ring is also Gorenstein. Indeed, if an algebra is Gorenstein, then the socle is mapped to zero in any nontrivial quotient of it. But the socle is mapped injectively under the cycle class map.

 We therefore study the tautological cohomology ring $RH^\bullet(\MM_{2,n})$; it suffices to prove that this ring is not Gorenstein. Since the whole cohomology ring $H^\bullet(\MM_{2,n})$ satisfies Poincar\'e duality, we see that in order for the tautological cohomology ring to be Gorenstein, any non-tautological cohomology class below the middle degree needs to be ``paired'' with another non-tautological cohomology class above the middle degree.  We prove that if $N$ is the smallest integer for which $\MM_{2,N}$ has a non-tautological cohomology class in even degree, then this class occurs in degree $2+N$ (so below the middle), and that no non-tautological class occurs in any other even degree. By \cite{graberpandharipande} we know that $N\leq 20$, and as we explain in the paper, we speculate that $N=20$ holds. We make use of recent work of Harder \cite{harder}, who has completely determined the Eisenstein cohomology of every local system on $\A_2$. 
  
In the process, we show that every cohomology class on $\MM_{2,20}$ of even degree which is pushed forward from the boundary is tautological, even though there is non-tautological cohomology in even degree on $\MM_{1,11}\times \MM_{1,11}$. This answers a question posed in \cite{graberpandharipande}. 

\subsection{Conventions} In Sections \ref{outline}, \ref{monodromy} and \ref{pc}, all cohomology means \'etale cohomology with $\Q_\ell$-coefficients, for some fixed prime $\ell$. The only properties we really use is the existence of a weight filtration satisfying the usual properties; our primary reason for working with \'etale cohomology is easy compatibility with the work of Harder. In Section \ref{localsystems} we shall have reason to switch between cohomology theories and we try to be more careful with coefficients. 


\section{An outline of the argument and the main theorems}
\label{outline}
Let $\M_{2,n}^\rt$ denote the moduli space of stable $n$-pointed curves of genus $2$ with \emph{rational tails}, i.e.\ which have an irreducible component of geometric genus $2$. Let $\widetilde{\partial \M_{2,n}^\rt}$ denote the normalization of the boundary of $\M_{2,n}^\rt$, so that every connected component of $\widetilde{\partial \M_{2,n}^\rt}$ is isomorphic to either $\MM_{1,n'}\times \MM_{1,n-n'}$ or $\MM_{1,n+2} / \sym_2$. Since $\widetilde{\partial \M_{2,n}^\rt}$ is a resolution of singularities of the boundary (in the sense of stacks), there is an exact sequence
\[ \label{fund} H^{k-2} (\widetilde{\partial \M_{2,n}^\rt})(-1) \stackrel {q_!}\to H^k(\MM_{2,n}) \to W_k H^k(\M_{2,n}^\rt) \to 0, \tag{$\partial$}\]
as one sees by combining \cite[Corollaire 8.2.8]{hodge3} and \cite[Corollaire 3.2.17]{hodge2}. (Deligne's arguments use only formal properties of the weight filtration.) 

\begin{lem}\label{lem1}Suppose that  all classes in the image of $q_!$ are tautological, and that all classes in $W_k H^k (\M_{2,n}^\rt)$ are tautological. Then all of $H^k (\MM_{2,n})$ is tautological. \end{lem}

\begin{proof}Clear. \end{proof}

By the definition of tautological ring, $q_!$ maps tautological classes to tautological classes.  By the first author's recent proof \cite{genusone} of the claims of \cite{ellipticgw}, we have a good understanding of precisely what classes in the domain of $q_!$ are tautological. (The second half of the following theorem is present already in \cite{getzler99, semiclassical}.)

\begin{thm}[Petersen, Getzler] All even cohomology classes on $\MM_{1,n}$ are tautological. All odd cohomology classes are Tate twists of Galois representations attached to cusp forms for $\SL_2(\Z)$, and the first example of this is the nonzero cohomology group $H^{11}(\MM_{1,11})$. \end{thm}

The argument used in \cite{genusone} was to apply a similar reduction as the one used in this section to reduce to the statement that $W_k H^k (\M_{1,n})$ is tautological when $k$ is even, or consists of cusp form classes when $k$ is odd. One then studies $H^k (\M_{1,n})$ and its weight filtration by fibering $\M_{1,n}$ over $\M_{1,1}$ and applying the Eichler--Shimura theory, which expresses the cohomology of local systems on $\M_{1,1}$ in terms of Galois representations attached to modular forms. 

\begin{cor}\label{cor1}Let $n < 20$, and suppose $k$ is even. Then $H^{k} (\widetilde{\partial \M_{2,n}^\rt})$ consists only of tautological classes. When $n=20$, this is false for $k=22$ and the cohomology of the connected components of the form $\MM_{1,11}\times \MM_{1,11}$. \end{cor}

However, even though there are non-tautological classes in even degree in $\widetilde{\partial \M_{2,20}^\rt}$, it turns out that none of them contribute nontrivially to the cohomology of $\MM_{2,20}$. In Section 5, we prove the following result.

\begin{thmb} Let $n=20$, and suppose $k$ is even. The image of the map $q_!$ in the sequence \emph{($\partial$)} consists only of tautological classes.   \end{thmb}

One should also understand the other end of the exact sequence, i.e.\ the cohomology $W_k H^k (\M_{2,n}^\rt)$. The strategy will be to fiber $\M_{2,n}^\rt$ over $\M_2$. This is why we found it useful to work specifically with the space of curves with rational tails: $\M_{2,n}^\rt \to \M_2$ is smooth and projective, and the Leray spectral sequence degenerates at $E_2$. An obstacle which appears at this point is that our knowledge of the cohomology of local systems in genus $2$ is more limited than in genus $1$. Nevertheless, we prove the following result:  

\begin{thma}Let $N$ be the smallest natural number such that  $W_{2i} H^{2i} (\M_{2,N}^\rt)$ contains a non-tautological class for some $i$. Then $N \in \{8,12,16,20\}$ and $i=\frac N 2 + 1$, and there are no non-tautological classes of the correct weight in other even degrees. \end{thma}

The existence of a non-tautological algebraic cycle on $\MM_{2,20}$ which restricts nontrivially to the interior \cite{graberpandharipande} shows that $N \leq 20$ in the preceding theorem. We give an alternative proof of this result in Sections \ref{localsystems} and \ref{monodromy}. 

\begin{cor}\label{corg}The tautological cohomology ring $RH^\bullet( \MM_{2,N})$ is not Gorenstein. \end{cor}

\begin{proof}Consider the exact sequence ($\partial$). By combining the results stated so far in this section, we see that when $k \neq N+2$ is even, then $H^{k}(\MM_{2,N})$ consists only of tautological classes, whereas there is a non-tautological class in $H^{N+2}(\MM_{2,N})$. But $H^{N+2}(\MM_{2,N})$ and $H^{N+4}(\MM_{2,N})$ are Poincar\'e dual to each other, so the tautological cohomology ring cannot satisfy Poincar\'e duality.  \end{proof}

As explained in the introduction, we deduce also the following result.

\begin{cor}The tautological ring $R^\bullet( \MM_{2,N})$ is not Gorenstein. \end{cor}

\begin{proof}Assume otherwise. Then $\phi \colon R^\bullet (\MM_{2,N} )\to RH^\bullet (\MM_{2,N})$ is not an isomorphism. Pick $\alpha$ in the kernel. Since we assumed $R^\bullet (\MM_{2,N})$ to be Gorenstein, we can pick some $\beta$ which pairs nontrivially with $\alpha$. But then $\phi(\alpha)\cup \phi(\beta) = \phi(\alpha \cup \beta) \neq 0$, since $\alpha \cup \beta$ is the socle, contradicting $\phi(\alpha)=0$. \end{proof}

\section{Cohomology of local systems}
\label{localsystems}
In \cite{genusone} a key point was to compute the lowest weight part of the cohomology of $\M_{1,n}$, and to do this, one needed first to know the lowest weight part of the cohomology of local systems on $\M_{1,1}$. More precisely, we consider natural local systems $\V_k$ of weight $k$ whose cohomology is well studied: one has
\[ W_{i+k} H^i(\M_{1,1},\V_k) \neq 0\]
only if $k=i=0$ or if $k \geq 10$ is even, $i=1$, and we are considering classes associated to cusp forms. Since we moreover only care about even cohomology of $\MM_{1,n}$, hence even weights, we see that only the trivial local system will contribute. Our goal in this section will be to carry out a similar analysis on $\M_2$ and $\A_2$: we will find that we need to consider only the local systems of the form $\V_{2a,2a}$, $a \geq 0$, whose definition we now recall. 

\subsection{The local systems}Let $\A_g$ be the moduli space of principally polarized abelian varieties of dimension $g$. The orbifold $\A_g(\mathbf C)$ is a locally symmetric hermitian space: it is isomorphic to $G(\Z) \backslash G(\mathbf R)/K$, where $G= \Sp_{2g}$ and $G(\mathbf R)/K$ is Siegel's upper half space. A rational representation of $G$ defines a (rational) local system on $\A_g$. Such representations are classified by their highest weights, which are $g$-tuples $\lambda = \lambda_1 \geq \cdots \geq \lambda_g \geq 0$. If all inequalities are strict, then the weight $\lambda$ and the corresponding local system $\V_\lambda$ are both called \emph{regular}. 

The sheaf $\V_\lambda \otimes \Q_\ell$ is naturally the base change of a local system (smooth $\ell$-adic sheaf) on $\A_g / \Q$, as explained in \cite{harderbook}, and its cohomology groups carry an action of $\mathrm{Gal}_{\overline \Q / \Q}$. The cohomology of the sheaf $\V_\lambda \otimes \mathbf C$ can be understood by transcendental methods. We will have reason to consider both the $\ell$-adic and the complex realizations of these sheaves, which are related by comparison isomorphisms in the usual way. 

The local systems $\V_\lambda$ should really be thought of as realizations of motivic sheaves. Concretely, they admit the following purely algebro-geometric description. Let $ \pi \colon \X \to \A_g$
be the universal abelian variety, and put $\V = \R^1 \pi_\ast \Q$; this is a rank $g$ local system with a symplectic pairing 
\[ \V \otimes \V \to \Q(-1). \]
(The Tate twist here indicates that we should really be working with $G = \mathrm{GSp}_{2g}$, but this will not really matter for our purposes.) Every irreducible representation of the symplectic group can be obtained from the standard representation by applying a symplectic Schur functor; we define the local system $\V_\lambda$, $\lambda \vdash n$, to be the image of the corresponding Schur functor applied to $\V^{\otimes n}$. 

We shall be interested in the cohomology groups of these local systems. The corresponding Galois representations carry a wealth of arithmetic information. A first observation is that the cohomology group $H^i(\A_g,\V_\lambda)$ can only be nonzero if the weight $|\lambda| = \lambda_1 + \ldots + \lambda_g$ is even: in the odd case, inversion on the universal abelian variety acts as multiplication by $-1$ on the fibers of $\V_\lambda$. A second observation is that since the Leray spectral sequence for the projection $\X^n \to \A_g$ from the $n$th fibered power of the universal abelian variety degenerates at $E_2$, the usual bounds on the weights of the cohomology of a nonsingular variety imply that 
\[ H^i(\A_g,\V_\lambda) \text{ resp. } H^i_c(\A_g,\V_\lambda) \]
are mixed Hodge structures/Galois representations of weight at least $i+|\lambda|$  (resp. at most $i+|\lambda|$). This can also be seen from general facts about the behavior of weights under the functors $\R f_!$ and $\R f_\ast$ \cite{weil2}.

In particular, if we only want to consider the lowest weight part and we want this weight to be even, then we can restrict our attention to $H^i(\A_g,\V_\lambda)$ where $i$ is even. 

\subsection{Eisenstein and inner cohomology} Consider the natural map $$\rho\colon H^i_c(\A_g,\V_\lambda) \to H^i(\A_g,\V_\lambda).$$ We define the \emph{Eisenstein cohomology} of $\V_\lambda$ to be the cokernel of $\rho$ and the \emph{inner cohomology} to be the image. These will be denoted by $H^i_\Eis(\A_g,\V_\lambda)$ and $H^i_!(\A_g,\V_\lambda)$, respectively. One can also define the compactly supported Eisenstein cohomology $H^i_{c,\Eis}$ as the kernel of $\rho$. The study of Eisenstein cohomology was initiated by Harder, see e.g. \cite{hardergl2,harderbook}. There is a short exact sequence \[ 0 \to H^i_!(\A_g,\V_\lambda) \to H^i(\A_g,\V_\lambda) \to H^i_\Eis(\A_g,\V_\lambda) \to 0. \]
The weight bounds quoted in the previous paragraph imply that $H^i_!(\A_g,\V_\lambda)$ is pure of weight $i+|\lambda|$. Since the sheaves $\V_\lambda$ are essentially self-dual, we also see that $H^\bullet_!(\A_g,\V_\lambda)$ satisfies Poincar\'e duality. One should think of the Eisenstein cohomology as cohomology ``at infinity'', associated to boundary strata in a suitable compactification of $\A_g$, and the inner cohomology as coming from the interior.  

\newcommand{\e}{\mathbf{e}}
\subsection{Conjectures via point counts}One can form a kind of ``motivic'' Euler characteristic of the local systems $\V_{\lambda}$. Let $\Gal$ be the abelian category of continuous $\ell$-adic representations of $\mathrm{Gal}_{\overline \Q / \Q}$, and define
\[ \e(\A_g, \V_\lambda) = \sum_{i \geq 0} (-1)^i [H^i_c(\A_g,\V_\lambda \otimes \Q_\ell)]\]
where the sum is taken in the Grothendieck group $K_0(\Gal)$. The Grothendieck--Lefschetz trace formula for stacks \cite{behrend93} and the interpretation of $\V_\lambda$ in terms of fibered powers of the universal abelian variety allows one to give a concrete interpretation of the trace of the Frobenius $\Frob_q$ on $\e(\A_g,\V_\lambda)$ as a sum over all isomorphism classes of principally polarized abelian varieties over $\field F_q$, weighted in a specific way by the eigenvalues of Frobenius on their first cohomology. Conversely, if one computes these traces of Frobenius by explicitly enumerating all {ppav}s over $\field F_q$, one can make precise conjectures regarding the value of $\e(\A_g,\V_\lambda)$ for every $\lambda$. This has been carried out for $g \leq 3$ in an ambitious project by Bergstr\"om, Faber and Van der Geer \cite{fvdg1,fvdg2,bfg08,bfg11}. The case $g=2$, which will be relevant for us, is well described in \cite[\S 6]{bfg11}.

\subsection{The case $g=2$}
Recently Harder \cite{harder} has completely determined the Eisenstein cohomology $H^i_\Eis(\A_2,\V_\lambda \otimes \Q_\ell)$ for any weight $\lambda$.  Harder's computations lead in particular to a strengthening of the conjectural formulas of Faber and Van der Geer: if one knows $H^i_\Eis(\A_2,\V_\lambda \otimes \Q_\ell)$ for all $i$ and assumes their formula for $\e(\A_2,\V_\lambda)$, then one immediately also knows the sum
\[ \sum_{i \geq 0} (-1)^i [H^i_!(\A_2,\V_\lambda \otimes \Q_\ell)] \]
in $K_0(\Gal)$ by subtraction. But $H^i_!$ satisfies purity, so we can read off each individual cohomology group by applying $\gr^W_{i+|\lambda|}$. Thus we have a conjectural formula for each cohomology group $H^i(\A_2,\V_\lambda \otimes \Q_\ell)$ considered as an $\ell$-adic Galois representation (up to semisimplification). A special case of this conjectural description, which still covers all cases relevant for us, is the following:

\begin{conjecture}\label{conj}Let $\lambda = \lambda_1 \geq \lambda_2 \geq 0$ be arbitrary, and suppose $i \neq 3$. Then $H^i_!(\A_2,\V_\lambda)$ vanishes. \end{conjecture}

This conjecture should be much weaker than the full conjectural formula: one expects the deeper arithmetic phenomena to occur in the middle dimension. If $\lambda$ is regular, then Conjecture \ref{conj} is well known in any genus: the map $H^\bullet_c(\A_g,\V_\lambda \otimes \mathbf C) \to H^\bullet(\A_g,\V_\lambda \otimes \mathbf C)$ always factors through the $L^2$-cohomology, and by a result of Faltings \cite{faltings83} the $L^2$-cohomology vanishes outside the middle degree for regular weights. A more high-powered proof uses the theorem that $H^i(\A_g,\V_\lambda)$ vanishes below the middle degree for regular weights, a consequence of Saper's theory of $\mathcal L$-modules \cite{lmodules}. 

We do not know how to prove Conjecture \ref{conj}. Perhaps this is known to the experts. However, we prove later in this section by slightly ad hoc methods that the conjecture holds for all local systems $\V_\lambda$ satisfying $|\lambda| \leq 20$, which is enough for us.

\subsection{The Gysin sequence}The Torelli map $\M_2 \to \A_2$ is an open immersion (which is special to genus two, in general it is a ramified double cover of its image). Its complement is the closed substack of products of elliptic curves; we write this substack as $\A_{1,1}$. We denote by $\V_{l,m}$ also the restrictions of the above local systems to $\M_2$ and to $\A_{1,1}$. We get a Gysin long exact sequence 
\[ H^{\bullet-2} (\A_{1,1},\V_\lambda)(-1) \to H^\bullet (\A_2,\V_\lambda) \to H^\bullet (\M_2,\V_\lambda) \to H^{\bullet-1}(\A_{1,1},\V_\lambda)(-1).\]
Many of these cohomology groups vanish. Both $\M_2$ and $\A_{1,1}$ are affine, which implies that their cohomology can be nonzero only in degree at most $3$ (resp. at most $2$). By Raghunathan's vanishing theorem \cite{raghunathan1}, the group $H^1(\A_2,\V_\lambda)$ vanishes. The aforementioned result of Saper gives vanishing below the middle degree for regular weights. 

Recall that we are really interested in the lowest weight part of the cohomology. We see from the Gysin sequence that $W_{k+|\lambda|} H^k(\A_2,\V_\lambda)$ surjects onto $W_{k+|\lambda|} H^k(\M_2,\V_\lambda)$. Moreover, as we only care about even weights, we can restrict ourselves to even $k$, hence to $k=2$ (on $\M_2$) and $k=2$ or $k=4$ (on $\A_2$). 

By Harder's results, we know the lowest weight part of Eisenstein cohomology on $\A_2$ in any degree:
\begin{thm}[Harder]\label{harder1}
Suppose $\lambda \neq (0,0)$. Then 
\[ W_{i+|\lambda|}H^i_\Eis(\A_2,\V_\lambda \otimes \Q_\ell) = \begin{cases}
\bigoplus_{\substack{f \in \Sigma_{4+4a} \\ L(f,2+a) \neq 0}} \Q_\ell(-1-2a) & \text{if $\lambda = (2a,2a)$, $i = 2$} \\
0 & \text{else.}
\end{cases}\]\end{thm}
Here $\Sigma_{k}$ denotes the set of normalized cusp eigenforms for $\SL_2(\Z)$ of weight $k$, and we include only those eigenforms whose $L$-function does not vanish at the point $2+a$. (It is not known whether there exists an eigenform whose $L$-function does vanish there. For weights up to $200$ there are none; tables can be found in William Stein's Modular Forms Database \cite{mfd}.) These are residual Eisenstein classes \cite{residual} coming from the Siegel parabolic subgroup. 

It is explained in \cite{d-elliptic} how to compute the cohomology groups $H^i(\A_{1,1},\V_{l,m})$: after one has determined the branching formula from $\Sp_4$ to $(\Sp_2)^2 \rtimes \sym_2$, the cohomology groups are expressed in terms of cohomology of local systems on $\A_1$. One finds the following result:

\begin{prop}\label{pca11}If $W_{i+|\lambda|}H^i(\A_{1,1},\V_\lambda) \neq 0$ and $i$ is even, then either $i=0$ and $\lambda = (2a,2a)$, or $i=2$ and $|\lambda| \geq 20$. In the former case we have $H^0(\A_{1,1},\V_{2a,2a}) = \Q(-2a)$. In the latter case we are considering a tensor product of two classes  in $H^1(\A_1,\V_l)$, $l \geq 10$, associated to cusp forms for $\SL_2(\Z)$. \end{prop}

We note in particular one consequence of what we have stated so far: although there are nonzero classes in $W_{2+|\lambda|}H^2(\A_{1,1},\V_{\lambda})$ for $|\lambda| \geq 20$, they all (conjecturally) vanish under the Gysin map 
\[ H^2(\A_{1,1},\V_\lambda)(-1) \to H^4(\A_{2},\V_\lambda)\]
since 
 Theorem \ref{harder1} and Conjecture \ref{conj} imply that there are no classes of the correct weight in $H^4(\A_2,\V_\lambda)$. In Section \ref{pc} we shall see how this implies that the Gysin map $H^{k}(\MM_{1,11}\times \MM_{1,11})(-1) \to H^{k+2}(\MM_{2,20})$ vanishes on $H^{11}(\MM_{1,11})\otimes H^{11}(\MM_{1,11})$. 

We also deduce the following result.

\begin{prop}\label{onm2}The lowest weight piece $W_{22} H^2 (\M_2,\V_{10,10})$ does not vanish. \end{prop}

\begin{proof}Consider the exact sequence
\[ H^0(\A_{1,1},\V_{10,10})(-1) \to H^2(\A_2,\V_{10,10}) \to H^2(\M_2,\V_{10,10}),\]
and apply $\gr^W_{22}$. Then the term in the middle becomes 2-dimensional by Theorem \ref{harder1} (and the vanishing of inner cohomology for weights at most 20, Theorem \ref{conj20}), since there are two distinct cusp eigenforms for $\SL_2(\Z)$ of weight $24$. However, the term to the right is 1-dimensional by Proposition \ref{pca11}. The result follows. 
\end{proof}

It is natural to believe that $\lambda = {(10,10)}$ is the ``first'' example of a nontrivial weight for which $W_{2k+|\lambda|} H^{2k}(\M_2,\V_\lambda) \neq 0$ for some $k$. Indeed the Gysin sequence shows that $W_{2k+|\lambda|} H^{2k}(\A_2,\V_\lambda)$ surjects onto $W_{2k+|\lambda|} H^{2k}(\M_2,\V_\lambda)$. The former consists only of Eisenstein cohomology according to Conjecture \ref{conj}, so by Theorem \ref{harder1} we need only to consider local systems of the form $\lambda = (2a,2a)$. When $a < 5$, $W_{2+4a} H^2_\Eis(\A_2,\V_{2a,2a})$ is at most $1$-dimensional. It is natural to believe that the map $$H^0(\A_{1,1},\V_{2a,2a})(-1) \cong \Q(-1-2a) \to H^2(\A_2,\V_{2a,2a}) $$ has as large rank as possible, i.e.\ that the map is nonzero as soon as $H^2(\A_2,\V_{2a,2a})$ contains a summand $\Q(-1-2a)$. By Harder's calculations this holds precisely when $a \neq 1$. This follows a very heuristic principle that ``any differential which is not zero for obvious reasons (weight, Galois action, etc.) should have as large rank as possible''. This would imply that we get something nonzero on $\M_2$ for the first time when $\lambda = (10,10)$. This result would also fit very well with the work of Graber and Pandharipande. Thus one is led to believe the following.

\begin{speculation}The map $H^0(\A_{1,1},\V_{2a,2a})(-1) \to H^2(\A_2,\V_{2a,2a})$ is injective for all natural numbers $a \neq 1$. \end{speculation}

Another way to phrase this speculation is that the local system $\V_{2,2}$ is the only local system on $\M_2$ for which $H^1(\M_2,\V_\lambda)$ is nonzero. The fact that $H^1(\M_2,\V_{2,2}) \neq 0$ has been verified by very different means in recent work in preparation by the second author \cite{tommasipreparation}. This corrects an error in \cite[Proposition 19]{getzlergenustwo}.

\subsection{Vanishing of inner cohomology}
We now prove that the inner cohomology does in fact vanish outside the middle degree for local systems of weight up to $20$, which will be enough for the applications in this article.

\begin{thm}\label{conj20}Suppose $|\lambda| \leq 20$. Then $H^i_!(\A_2,\V_\lambda)$ vanishes for $i \neq 3$. \end{thm}

\begin{proof}We know that for a nontrivial local system we have cohomology only in degrees $2$,$3$ and $4$, so by Poincar\'e duality of the inner cohomology it suffices to prove that $H^4_!(\A_2,\V_\lambda)$ vanishes. 

First consider the exact sequence
\[  H^2(\A_{1,1},\V_\lambda)(-1) \to H^4 (\A_2,\V_\lambda) \to 0 \]
coming from the Gysin sequence, and the fact that $\M_2$ is affine. This shows that $H^4_!(\A_2,\V_\lambda)$ is a quotient of $W_{2+|\lambda|} H^2(\A_{1,1},\V_\lambda)(-1)$. The latter space vanishes when $|\lambda| < 20$ by Proposition \ref{pca11}, so $H^i_!(\A_2,\V_\lambda)$ vanishes automatically outside the middle degree for these weights. 

For $|\lambda|=20$ it suffices to consider $\V_{20,0}$ and $\V_{10,10}$, since as explained earlier we know vanishing for regular weights by \cite{faltings83}. For both of these local systems one checks that 
\[ W_{2+|\lambda|} H^2 (\A_{1,1},\V_\lambda) = \wedge^2 S[12],\]
where $S[12]$ is the motive attached to cusp forms of weight $12$ for $\SL_2(\Z)$ by Deligne--Scholl. The motive $\wedge^2 S[12]$ is just the Tate motive of weight $22$; its $\ell$-adic/Hodge--theoretic realizations are $\Q_\ell(-11)$ and $\Q(-11)$, respectively. 

For the local system $\V_{20,0}$ one can apply the description of the Hodge filtration on $H^\bullet(\A_g,\V_\lambda \otimes \mathbf C)$ via the BGG-complex due to Faltings and Chai \cite[VI.5.5.]{faltingschai}. Their construction implies that the cohomology groups of $\V_{a,b}$ have a Hodge filtration where all $F$-weights lie in the set $\{a+b+3,a+2,b+1,0\}$. In particular, there can be no classes which are pure of Tate type in the cohomology of a local system of the form $\V_{a,0}$. This shows that $H^2_!(\A_2,\V_{20,0})$ must in fact vanish.


Thus it remains to consider the local system $\V_{10,10}$. The conjectures of Faber and Van der Geer imply that the inner cohomology of $\V_{10,10}$ will in fact vanish in every degree, including the middle. Now one also knows the integer-valued Euler characteristic of the inner cohomology, i.e. the number
\[ \chi_!(\A_2,\V_\lambda) = \sum_{i} (-1)^i \dim H^i_!(\A_{2},\V_\lambda)\]
for any local system $\lambda$: the Euler characteristic on $\M_2$ of all local systems $\V_\lambda$ were determined by Getzler \cite{getzler02}, on $\A_{1,1}$ one uses \cite{d-elliptic}, and we can subtract off the Euler characteristic of the Eisenstein cohomology using Harder's results. So we know in fact unconditionally that $\chi_!(\A_2,\V_{10,10}) = 0$. 

If it really were the case that $H^2_!(\A_2,\V_{10,10})$ did not vanish, then the ranks of $H^i_!(\A_2,\V_{10,10})$ would need to be $1,2,1$ in degrees $2,3$ and $4$, respectively (to match up with Poincar\'e duality and the Euler characteristic). By Poincar\'e duality we would also know that $H^2_!$ and $H^4_!$ are given by $\Q_\ell(-11)$ and $\Q_\ell(-12)$ in the $\ell$-adic realization.  Hence if $\alpha_1$ and $\alpha_2$ denote the eigenvalues of $\Frob_q$ on $H^3_!(\A_2,\V_{10,10} \otimes \Q_\ell)$, then 
\[ \mathrm{Tr}\left(\Frob_q \mid H^\bullet_!(\A_2,\V_{10,10} \otimes \Q_\ell) \right)= q^{11} - \alpha_1 - \alpha_2 + q^{12}.\]We claim that this number can not vanish for any $q$. Indeed we have $|\alpha_i| = q^{23/2}$ for $i=1,2$, so the triangle inequality implies that 
\begin{align*}  
|q^{11} - \alpha_1 - \alpha_2 + q^{12}| & \geq |q^{11}+q^{12}| - |\alpha_1| - |\alpha_2| \\
& = q^{11}+q^{12} - 2q^{23/2} \\
& = q^{11}(1-2\sqrt q + q) \\
&= q^{11}(1-\sqrt q)^2  > 0. \end{align*}
But the conjectures of Faber and Van der Geer predict that the inner cohomology of $\V_{10,10}$ should vanish, in particular they have computed that the trace of Frobenius on the inner cohomology vanishes for all primes up to $37$. This contradiction proves that the inner cohomology could not have been nonzero in degree $4$ for the local system $\V_{10,10}$, either. 
\end{proof}

\begin{rem}A more careful version of the argument given above using the results of Faltings--Chai will in fact prove the result for all local systems of the form $\V_{a,0}$; the difficult case seems to be the local systems of the form $\V_{a,a}$. \end{rem}

\section{Monodromy invariants in the fibers of $\M_{2,n}\to \M_2$}
\label{monodromy}
Let $X$ be a smooth projective curve of genus $g \geq 2$. We shall consider its $n$-fold product $X^n$, and the \emph{Fulton--MacPherson compactification} \cite{fmcompactification} $\FM(X,n)$ of the configuration space of points on $X$. The \emph{tautological ring} of $X^n$  is defined as the subalgebra of its rational Chow ring generated by the classes of the $\binom n 2$ diagonals $\Delta_{ij}$, and the classes $K_i$ which are given by the canonical divisor on the $i$th factor. The tautological ring of $\FM(X,n)$ is defined analogously, except one moreover takes as generators the boundary strata of the compactification. As for the tautological rings of moduli spaces of curves, one can also consider the image of the tautological ring in cohomology. In this case, each class $K_i$ is proportional to the class of a point $a_i$, and we can instead take the classes of points as generators. We denote these rings by $RH^\bullet (X^n)$ and $RH^\bullet (\FM(X,n))$, respectively. The tautological ring of a single curve was introduced by Faber and Pandharipande, see \cite[\S 2]{faberjapan}.

\begin{prop}Let $X$ be a smooth projective curve of genus $g$. The subalgebra $H^\bullet(X^n)^{\Sp_{2g}}$ of monodromy invariant classes coincides with the tautological cohomology ring $RH^\bullet(X^n)$.  \end{prop}

\begin{proof}By the K\"unneth theorem, $H^\bullet(X^n)$ is a sum of terms of the form $$H^2(X)^{\otimes k} \otimes H^1(X)^{\otimes l}.$$ Clearly, $H^2(X)^{\otimes k}$ is $\Sp_{2g}$-invariant and spanned by a tautological class. Now consider $H^1(X)^{\otimes l}$. The decomposition of this space into irreducible representations of $\Sp_{2g}$ can be understood from Weyl's construction of the irreducible representations of $\Sp_{2g}$, \cite[\S 17.3]{fh91}. In particular we see from loc.\ cit.\ that a $\Sp_{2g}$-invariant class necessarily can be obtained from a tensor in $H^1(X)^{\otimes (l-2)}$ by inserting the class of the symplectic pairing at two indices $i$ and $j$. But this is the same as multiplying by $\Delta_{ij} - a_i - a_j$, so inductively we see that a $\Sp_{2g}$-invariant class is tautological. The reverse inclusion is clear. \ \end{proof}

\begin{prop}The subalgebra $H^\bullet(\FM(X,n))^{\Sp_{2g}}$ of monodromy invariant classes coincides with the tautological cohomology ring $RH^\bullet (\FM(X,n))$.  \end{prop}

\begin{proof}The projection $\FM(X,n) \to X^n$ makes $H^\bullet (\FM(X,n))$ an algebra over $H^\bullet (X^n)$, generated by classes $D_S$ associated to strata of $\FM(X,n)$ defined by certain points coinciding, see \cite[Corollary 7a]{fmcompactification}. The natural $\Sp_{2g}$-action on $H^\bullet (\FM(X,n))$ is compatible with this algebra structure, and can be defined by declaring that each generator $D_S$ is $\Sp_{2g}$-invariant. Thus an $\Sp_{2g}$-invariant class in $H^\bullet (\FM(X,n))$ can be written as the product of an $\Sp_{2g}$-invariant class in $H^\bullet (X^n)$ and a product of classes $D_S$; in particular, any $\Sp_{2g}$-invariant class is tautological by the previous proposition. \end{proof}

\begin{thm}Let $X$ be a curve of genus two. There is an isomorphism $RH^\bullet (\M_{2,n}^{rt}) \cong RH^\bullet (\FM(X,n))$. \end{thm}

\begin{proof}Let $\pi \colon \M_{2,n}^{rt} \to \M_2$ be the projection, and let us consider the Leray spectral sequence for $\pi$, which degenerates because $\pi$ is smooth and proper. The fiber of $\pi$ over $[X]$ is exactly $\FM(X,n)$. The higher pushforwards of $\pi$ are local systems, defined by the action of $\Sp_4$ on the cohomology of $\FM(X,n)$. Since $\M_2$ has the rational cohomology of a point, we have an isomorphism
\[ H^0(\M_2,\R^\bullet \pi_\ast\Q_\ell) \cong H^\bullet(\FM(X,n))^{\Sp_4} \cong RH^\bullet (\FM(X,n)).\]
On the other hand, every tautological class on $\M_{2,n}^{rt}$ is monodromy invariant (being  the class of a stratum possibly multiplied with a $\psi$-class), so all classes in the kernel of $H^\bullet( \M_{2,n}^{rt}) \to H^\bullet (\FM(X,n))$ are nontautological. The result follows. \end{proof}

\begin{rem}In particular, we have given a conceptual proof that all these tautological rings are Gorenstein. Indeed both $H^\bullet (X^n)$ and $H^\bullet (\FM(X,n))$ satisfy Poincar\'e duality, and since the cup product is compatible with the $\Sp_{2g}$-action, we see that the subring of monodromy invariants does, too. When $g=2$ this result is not new, however: it was proven by Tavakol \cite{tavakol2} in Chow, which implies the result also in cohomology. The fact that $RH^\bullet (X^n)$ is Gorenstein was announced without proof in \cite[\S 2]{faberjapan}. The unpublished proof of this fact is rather different from ours and uses instead an approach similar to that used Tavakol in genus 1 and 2 in \cite{tavakol1, tavakol2}. In Chow, it is not true in general that the tautological ring of $X^n$ is Gorenstein: a counterexample is constructed for $g=4$ and $n=2$ in \cite{greengriffiths}. \end{rem}

Let us now determine the pure part of the even cohomology of $\M_{2,n}^{rt}$, i.e.\ the lowest weight pieces $W_{2i} H^{2i}(\M_{2,n}^{rt})$. We shall do this by means of the Leray spectral sequence for $\pi \colon \M_{2,n}^\rt \to \M_2$. Hence we must determine the sheaves $\R^q \pi_\ast \Q_\ell$, or what amounts to the same thing, the action of $\Sp_{4}$ on $H^\bullet (\FM(X,n))$. 

The decomposition of $H^\bullet (X^n)$ into irreducible representations is easily read off from Weyl's construction of the irreducible representations of the symplectic group, as mentioned earlier. Consider a highest weight $l \geq m \geq 0$ with $l+m = n$. The corresponding irreducible representation appears for the first time in the cohomology of $X^{n}$, where it occurs in the middle degree $H^n (X^n)$ (in the image of $H^1(X)^{\otimes n} \subset H^n (X^n)$) and nowhere else.  

Using Fulton and MacPherson's explicit description of the cohomology of $\FM(X,n)$ it is not hard to extend this analysis also to $\FM(X,n)$: under the map $H^n(X^n) \hookrightarrow H^n(\FM(X,n))$ we see that every irreducible representation of weight $(l,m)$ with $l+m = n$ occurs in the middle degree also in $\FM(X,n)$, and the fact that it does not in any other degrees follows from relation (2)(i) of \cite[Corollary 7a]{fmcompactification}. We have proven the following result:

\begin{prop}\label{proppen}Let $\pi \colon \M_{2,n}^\rt \to \M_2$, and let $l+m = n$. The local system $\V_{l,m}$ occurs as a summand in $\R^n \pi_\ast \Q_\ell$ but not in any other degrees. \end{prop}

\begin{rem}A slightly more careful reading of Weyl's construction will give also the multiplicity of $\V_{l,m}$ and its decomposition as an $\sym_n$-representation. For instance, the local system $\V_{k,k}$, where $n = 2k$, is very relevant for us: it occurs once, tensored with the irreducible representation of $\sym_{2k}$ corresponding to the partition $(2,2,\ldots,2)$. It is not hard to see that the non-tautological algebraic cycle on $\MM_{2,20}$ constructed in \cite{graberpandharipande} transforms according to the representation $(2,2,\ldots,2)$, so their results fit well with ours, and our belief that the local system $\V_{10,10}$ gives rise to the first nontautological class on $\MM_{2,n}$. \end{rem}

Now let $N$ be the smallest positive integer for which there is a local system $\V_{l,m}$, $l+m = N$, and an $i$ for which 
\[ W_{2+N} H^{2} (\M_2,\V_{l,m}) \neq 0.\]
As we shall see very soon, this $N$ is the same as the one defined in the introduction. By Propositions \ref{onm2} and \ref{conj20} and the fact that $W_{2k+|\lambda|} H^{2k}(\A_2,\V_\lambda)$ surjects onto $W_{2k+|\lambda|} H^{2k}(\M_2,\V_\lambda)$, we have $N \in \{8,12,16,20\}$ and $l=m=N/2$. Since the Leray spectral sequence for $\pi$ degenerates, we see as a consequence of Proposition \ref{proppen} that we find a nonvanishing cohomology class in $W_{2+N} H^{2+N} (\M_{2,N}^{rt})$ which is \emph{not} monodromy invariant, i.e.\ which vanishes under the restriction map $H^\bullet(\M_{2,N}^\rt) \to H^\bullet (\FM(X,N))$, and therefore is nontautological. Since the local system $\V_{l,m}$ only occurred in $\R^N \pi_\ast\Q_\ell$, there can be no nontautological classes of the correct  weight in other even degrees. We have therefore proven the following theorem stated in Section \ref{outline}:

\begin{thm}\label{aa}Let $N$ be the smallest natural number such that  $W_{2i} H^{2i} (\M_{2,N}^\rt)$ contains a non-tautological class for some $i$. Then $N \in \{8,12,16,20\}$ and $i=\frac N 2 + 1$, and there are no non-tautological classes of the correct weight in other even degrees. \end{thm}

\renewcommand{\Q}{\mathbf Q_\ell}
\section{Vanishing of PC classes}
\label{pc}
All that remains at this point is to prove Theorem \ref{bb} announced in Section \ref{outline}. So let $n=20$, and let $k$ be even. We should prove that the image of $q_!$ in the exact sequence $(\partial)$ consists only of tautological classes. As in Corollary \ref{cor1}, all classes in $\widetilde \M_{2,20}^\rt$ are tautological except those  cohomology classes on $\MM_{1,11}\times \MM_{1,11}$ which are given by tensor products of two classes in $H^{11}(\MM_{1,11})$ associated to cusp forms. We call these \emph{PC classes} for brevity (``\emph{P}roducts of \emph{C}usp forms''). 

Our goal will be to prove that all PC classes are mapped to zero under $q_!$ in the exact sequence ($\partial$). It will at this point become more convenient to switch to compactly supported cohomology: dually, it is equivalent to prove that they are not in the image of the map $H^{22}_c(\MM_{2,20}) \to H^{22}_c(\MM_{1,11}\times \MM_{1,11})$. This, in turn, amounts to the same as proving that they are not in the image of the restriction map $H^{22}_c(\MM_{2,20}) \to H^{22}_c(\partial \M_{2,20})$, since the PC classes on $\MM_{1,11}\times \MM_{1,11}$ are pulled back from $\partial \M_{2,20}$. To see this last fact, use that the PC classes are supported on the interior, i.e.\ in the image of $H^{11}_c(\M_{1,11}) \to H^{11}_c(\MM_{1,11})$. The same applies then to $H^{22}_c(\M_{1,11}\times \M_{1,11}) \to H^{22}_c(\partial \M_{2,20})$.

Let $W \subset \MM_{2,20}$ be the union of the open stratum $\M_{2,20}$ and all strata of the form $\M_{1,11}\times \M_{1,11}$. There is a cartesian diagram
\[ \begin{diagram}
\partial \M_{2,20} & \lTo & \coprod \M_{1,11}\times \M_{1,11} \\
\dTo & & \dTo \\
\MM_{2,20} & \lTo & W
\end{diagram}\]
whose horizontal (resp.\ vertical) rows are open (resp.\ closed) immersions. By the functoriality of the exact sequence of a pair in compactly supported cohomology, we get a commutative diagram with exact rows
\begin{diagram}
H^k_c  (\partial W) & \lTo & H^k_c(\partial \M_{2,20}) & \lTo & \bigoplus H^k_c(\M_{1,11}\times \M_{1,11}) \\
\uEquals && \uTo_{\alpha^\ast} & & \uTo_{\beta^\ast} \\
H^k_c  (\partial W) & \lTo & H^k_c(\MM_{2,20}) & \lTo & H^k_c(W),
\end{diagram}
where $\partial W = \MM_{2,20} \setminus W$. 
We wish to prove here that the PC classes are not in the image of $\alpha^\ast$. Using that the PC classes go to zero in $H^k_c(\partial W)$, a very easy diagram chase shows it suffices to show that they are not in the image of $\beta^\ast$. 

\begin{rem}The arguments in the previous paragraph can be seen as computing by hand in the spectral sequence associated to the three-step filtration $\M_{2,20} \subset W \subset \MM_{2,20}$. \end{rem}

Now let $\C$ be the universal curve over $\M_{2}^\ct \cong \A_2$, and $\C^n$ its $n$th fibered power. Let similarly $\D$ be the universal curve over $\M_2^{\ct} \setminus \M_2 \cong \A_{1,1}$. There is a cartesian diagram
\[\begin{diagram}
\coprod \M_{1,11}\times \M_{1,11} & \rTo & W \\
\dTo && \dTo \\
\D^{20} & \rTo & \C^{20} \\
\end{diagram}\]

where the horizontal (resp.\ vertical) arrows are closed (resp.\ open) immersions. Since classes of the correct weight are mapped injectively under open immersions, and the PC classes do have the correct weight, it will then suffice to show that their image in the cohomology of $\D^{20}$ is not in the image of the map $H^k_c(\C^{20}) \to H^{k}_c(\D^{20})$. 

Finally, we consider the diagram 
\[\begin{diagram}
\D^{20} & \rTo & \C^{20} \\
\dTo & & \dTo_\pi \\
\A_{1,1} &\rTo^j& \A_2.
\end{diagram}\]

Here we have done the identifications $\A_{1,1} \cong \Sym^2 \M_{1,1}$ and $\A_2 \cong \M_2^\ct$. If we consider the Leray spectral sequence for $f \colon \M_{1,11} \to \M_{1,1}$, then the cusp form classes appear in $H^1_c(\M_{1,1},\R^{10}f_!\Q)$. Hence if we instead consider $g \colon \D^{20} \to \A_{1,1}$, then the PC classes appear in $H^2_c(\A_{1,1},\R^{20}g_! \Q)$. By functoriality of the Leray spectral sequence, we finally see that it will be enough to prove that the map 
\[ \gr^W_{22} H^2_c(\A_2,\R^{20} \pi_! \Q) \to  \gr^W_{22} H^2_c(\A_{1,1},j^\ast \R^{20} \pi_! \Q) \]
vanishes. (By the proper base change theorem we may identify $j^\ast \R^\bullet \pi_!\Q$ with the derived pushforward of $\D^{20} \to \A_{1,1}$.)

\begin{lem} If $\rho \colon \C \to \A_2$ is the projection of the universal curve, then 
\begin{align*}
\R^0 \rho_! \Q &= \Q, \\
\R^1 \rho_! \Q &= \V, \\
\R^2 \rho_! \Q &= \Q (-1) \oplus j_\ast \U(-1).
\end{align*}
Here $\U$ denotes the local system on $\A_{1,1}$ corresponding to the sign representation; see \cite{d-elliptic} for more explanation of the notation. 
\end{lem}

\begin{proof}The easiest way to see this is perhaps via the decomposition theorem of \cite{bbd}. The total pushforward is a direct sum of of simple perverse sheaves, which are (up to a degree shift) intermediate extensions of local systems on those locally closed subvarieties on which $\rho$ is a topologically locally trivial fibration. Over $\M_2$, the total pushforward is $\Q \oplus \V[1] \oplus \Q(-1)[2]$. All these summands are restrictions of local systems on $\A_2$, which are thus the intermediate extensions. For the term $j_\ast \U(-1)$, note that $\R^2 \rho_!\Q$ restricted to $\A_{1,1}$ is 2-dimensional; it is the sum of an $\sym_2$-invariant local system (generated by the sum of the fundamental classes of the two components) and an anti-invariant part (generated by their difference). So on $\A_{1,1}$ we find also the summand $\U$, and its intermediate extension is just given by $j_\ast$ since $\A_{1,1}$ is closed.  \end{proof}

\begin{rem}One can also see these facts without use of the decomposition theorem; for example, the fact that $\R^1\rho_!\Q$ is a local system on all of $\M_2^\ct$ can be understood transcendentally from the fact that the monodromy around a component of $\M_2^\ct \setminus \M_2$ is given by a Dehn twist around a separating curve, which acts trivially on the cohomology of the surface \cite[6.5.2]{farbmargalit}. \end{rem}

After the preceding lemma, the sheaf $\R^{20}\pi_! \Q$ is determined by the K\"unneth theorem. When we compute $\R^{20}\pi_! \Q$, we find a sum with two kinds of terms: (i) those which include a factor $j_\ast \U(-1)$, and (ii) those which do not.

First consider case (i), i.e.\ a summand $\F$ of $\R^{20}\pi_! \Q$ which includes a factor $j_\ast \U(-1)$. In this case, the map 
\[ H^2_c (\A_2, \F) \to H^2_c(\A_{1,1},j^\ast \F)\]
is always an isomorphism. However, we claim that $\gr^W_{22} H^2_c(\A_{1,1}, j^\ast \F)$ always vanishes in this case. Indeed $j^\ast \F$ is a local system on $\A_{1,1}$ which is pure of weight $20$. But since it includes a factor $\U(-1)$, it is in fact a Tate twist of a local system of weight at most $18$, and the computation of the cohomology of local systems on $\A_{1,1}$ (Theorem \ref{pca11}) shows that there are no classes in $H^2_c$ of the correct weight for local systems of weight below $20$; in fact, the first classes of the correct weight in $H^2_c$ are exactly the tensor products of cusp forms. 

We move on to case (ii), in which case the summand $\F$ is a local system on $\A_2$. In this case, we have 
\[ \gr^W_{22} H^2_c(\A_2,\F) =  H^2_!(\A_2, \F) \oplus \gr^W_{22} H^2_{c, \Eis}(\A_2,\F).\]
We see from Harder's computations (Theorem \ref{harder1}) that the latter summand always vanishes. According to Theorem \ref{conj20}, the first summand does, too. All in all, we deduce the following result:

\begin{thm} \label{bb}Let $n=20$, and suppose $k$ is even. The image of the map $q_!$ in the sequence \emph{($\partial$)} consists only of tautological classes.   \end{thm}

\bibliographystyle{alpha}
\bibliography{../database}

\end{document}